\documentclass{amsart} 
\usepackage{amsmath,amssymb,amsthm}
\usepackage{hyperref}
\usepackage{xcolor}

\calclayout

\newcommand{\la}{\lambda}

\newcommand\ep{\varepsilon}

\newcommand\ga{\gamma}
\renewcommand{\phi}{\varphi}

\newcommand{\R}{{\mathbb R}}

\newcommand{\C}{{\mathbb C}} 
\newcommand{\N}{{\mathbb N}} 

\newcommand{\be}{\begin{equation}}
\newcommand{\ee}{\end{equation}}

\newcommand{\co}{\colon}

\newcommand{\G}{\Gamma}

\newcommand{\KM}{K^{}_{\!M}}
\newcommand{\KE}{K^{}_{\!E}}

\newcommand{\vol}{\operatorname{vol}}

\newcommand{\tr}{\operatorname{trace}}
\newcommand{\Hom}{\operatorname{Hom}}

\def\til{\widetilde}
\def\wt{\widetilde}
\def \bfo {\begin {eqnarray*}}
\def \efo {\end {eqnarray*} }
\def \beq {\begin {eqnarray}}
\def \eeq {\end {eqnarray}}

\def\tprod#1{\langle\!\langle #1\rangle\!\rangle}
\def\tnorm#1{|\!|#1|\!|}

\theoremstyle{definition}
\newtheorem{definition}{Definition}[section]
\newtheorem{notation}[definition]{Notation}

\newtheorem{remark}[definition]{Remark}

\theoremstyle{plain}

\newtheorem{theorem}{Theorem} 
\newtheorem{lemma}[definition]{Lemma} 
\newtheorem*{sublemma}{Sublemma} 
 
\newtheorem{proposition}[definition]{Proposition} 
\newtheorem{corollary}[definition]{Corollary}

\numberwithin{equation}{section}

\begin{document} 

\title{Approximations of the connection Laplacian spectra}

\author[D. Burago]{Dmitri Burago}                                                          
\address{Dmitri Burago: Pennsylvania State University,                          
Department of Mathematics, University Park, PA 16802, USA}                      
\email{burago@math.psu.edu}                                                     
                                                                                
\author[S. Ivanov]{Sergei Ivanov}
\address{Sergei Ivanov:
St.Petersburg Department of Steklov Mathematical Institute,
Russian Academy of Sciences,
Fontanka 27, St.Petersburg 191023, Russia}
\email{svivanov@pdmi.ras.ru}

\author[Y. Kurylev]{\framebox{Yaroslav Kurylev}}        
\address{Yaroslav Kurylev}                                                  

\author[J. Lu]{Jinpeng Lu}
\address{Jinpeng Lu: Department of Mathematics and Statistics, University of Helsinki, FI-00014 Helsinki, Finland} 
\email{jinpeng.lu@helsinki.fi}


\thanks{The first author was partially supported
by NSF grant DMS-1205597.
The second author was partially supported by
RFBR grant 20-01-00070.
The fourth author was partially supported by Finnish Centre of Excellence in Inverse Modelling and Imaging.
}



\begin{abstract}
We consider a convolution-type operator on vector bundles over metric-measure spaces. This operator extends the analogous convolution Laplacian on functions in our earlier work to vector bundles, and is a natural extension of the graph connection Laplacian. We prove that for Euclidean or Hermitian connections on closed Riemannian manifolds, the spectrum of this operator and that of the graph connection Laplacian both approximate the spectrum of the connection Laplacian.
\end{abstract}

\date{}

\maketitle 

\section{Introduction}
\label{sec:introduction}

This paper is a continuation of our previous works where we approximated, in the spectral sense, the Riemannian Laplace-Beltrami operator with the discrete graph Laplacian \cite{BIK1,L} and a convolution-type operator \cite{BIK2}. This convolution-type operator, called the \emph{$\rho$-Laplacian} (with a small parameter $\rho>0$), is defined by averaging over metric balls of small radius, and it is a natural extension of the discrete graph Laplacian in a continuous setting. A notable feature of the $\rho$-Laplacian is that it is not based on differentiation and is readily available on general metric-measure spaces. Furthermore, we proved in \cite{BIK2} that the spectrum of the $\rho$-Laplacian enjoys stability under metric-measure approximations in a large class of metric-measure spaces. Ideally, we define the $\rho$-Laplacian as a notion of Laplacian on metric-measure spaces (in the spectral sense). Our earlier results in \cite{BIK1,BIK2,L} show that the definition makes sense for Riemannian manifolds. We hope that the spectra of the $\rho$-Laplacians could converge as $\rho\to 0$, in a large class of metric-measure spaces, with the limit related to known concepts of Laplacian in \cite{G1,G2}.

The present paper is concerned with the connection Laplacian on vector bundles. In this paper, we introduce an analogous convolution Laplacian acting on vector bundles over metric-measure spaces. This operator can be regarded as a generalization of the $\rho$-Laplacian (on functions), and its discretization, also known as the \emph{graph connection Laplacian}, is a generalization of the graph Laplacian. We prove that for Euclidean or Hermitian connections on closed Riemannian manifolds, our convolution Laplacian and its discretization both approximate the standard connection Laplacian in the spectral sense. The spectral convergence of the graph connection Laplacians may have applications in numerical computations and manifold learning, in particular analyzing high-dimensional data sets, see e.g. \cite{BN,TGHS,JH,KW,SW,SW2,WR} and the references therein.

\medskip
In this introduction, we define our operator for vector bundles over Riemannian manifolds. The general definition for metric-measure spaces can be found in Section \ref{sec:general}. 
Let $M^n$ be a compact, connected Riemannian manifold of dimension $n$ without boundary, and let $E$ be a smooth Euclidean (or Hermitian) vector bundle over $M$ equipped with a smooth Euclidean (or Hermitian) connection $\nabla$. 
Recall that an Euclidean (resp. Hermitian) connection is a connection that is compatible with the Euclidean (resp. Hermitian) metric on the vector bundle.
We denote by $L^2(M,E)$ the space of $L^2$-sections of the vector bundle $E$, and by $E_x$ the fiber over a point $x\in M$.
Fix $\rho>0$ smaller than the injectivity radius $r_{inj}(M)$.
Given any pair of points $x,y\in M$ with $d(x,y)\leq\rho$, let $P_{xy}:E_y\to E_x$ be the parallel transport canonically associated with $\nabla$ from $y$ to $x$ along the unique minimizing geodesic $[yx]$. 

For an $L^2$-section $u\in L^2(M,E)$, we define the \emph{$\rho$-connection Laplacian} operator $\Delta^{\rho}$ by
\begin{equation}\label{sec1-def-rhoLaplacian}
\Delta^\rho u(x)
  = \frac{2(n+2)}{\nu_n\rho^{n+2}} \int_{B_\rho(x)} \big(u(x)-P_{xy}(u(y)) \big)\,dy,
\end{equation}
where $\nu_n$ is the volume of the unit ball in $\mathbb{R}^n$, and $B_{\rho}(x)$ is the geodesic ball in $M$ of radius $\rho$ centered at $x\in M$.

The operator $\Delta^{\rho}$ is nonnegative and self-adjoint with respect to the standard inner product on $L^2(M,E)$. Furthermore, the lower part of the spectrum of $\Delta^{\rho}$ is discrete. We denote by $\til\la_k$ the $k$-th eigenvalue of $\Delta^{\rho}$ from the discrete part of the spectrum. Denote by $\Delta$ the standard connection Laplacian of the connection $\nabla$, and by $\la_k$ the $k$-th eigenvalue of $\Delta$. Our first result states that the spectrum of the $\rho$-connection Laplacian $\Delta^{\rho}$ approximates the spectrum of the connection Laplacian $\Delta$.

\begin{theorem}\label{t:smoothcase}
There exist constants $C_n>1$ and $c_n,\sigma_n\in (0,1]$, depending only on $n$, such that the following holds.
Suppose that the absolute value of the sectional curvatures of $M$ 
and the norm of the curvature tensor of $\nabla$ are bounded by
constants $\KM$ and $\KE$, respectively. 
Assume that $\rho>0$ satisfies
\be\label{e:rho-bound}
  \rho<\min \big\{ r_{inj}(M), c_n K_M^{-1/2} \big\}.
\ee
Then for every $k\in \N_+$ satisfying $\til\lambda_k\le \sigma_n \rho^{-2}$, we have
$$
  \big| \til\la_k^{1/2} - \la_k^{1/2} \big| \le
   \big(C_n\KM+\lambda_k \big)\lambda_k^{1/2} \rho^2 + C_n\KE\rho \, .
$$
\end{theorem}

\begin{remark}
One can track the dependence of $c_n$ on $n$ and find that
it suffices to assume
$$
 \frac{\sinh^{n-1} c_n}{\sin^{n-1} c_n} < 2.
$$
\end{remark}

\begin{remark}
In the case of $E$ being the trivial bundle, the connection Laplacian is simply the Laplace-Beltrami operator on functions, and the $\rho$-connection Laplacian $\Delta^{\rho}$ reduces to an operator on functions
$$\Delta^\rho f(x)
  = \frac{2(n+2)}{\nu_n\rho^{n+2}} \int_{B_\rho(x)} \big(f(x)-f(y) \big)\,dy.$$
This operator is the $\rho$-Laplacian we introduced in \cite{BIK2} up to a normalization adjustment, and its discretization is the graph Laplacian studied in \cite{BIK1}. In this case, Theorem \ref{t:smoothcase} reduces to the convergence of the spectra of $\rho$-Laplacians (on functions) to the spectrum of the Laplace-Beltrami operator, which is known from Theorem 1 in \cite{BIK1} and Theorem 1.2 in \cite{BIK2}.
\end{remark}

\smallskip
Now let us turn to the discrete side. We define a discretization of a compact Riemannian manifold $M$ as follows (see \cite{BIK1}).
\begin{definition}\label{def-graph}
Let $\varepsilon\ll\rho$ and $X_{\varepsilon}=\{x_i\}_{i=1}^N$ be a finite $\varepsilon$-net in $M$. The distance function on $X_{\varepsilon}$ is the Riemannian distance $d$ of $M$ restricted onto $X_{\varepsilon}\times X_{\varepsilon}$, denoted by $d|_{X_{\varepsilon}}$. Suppose that $X_{\varepsilon}$ is equipped with a discrete measure $\mu=\sum_{i=1}^N \mu_i\delta_{x_i}$ which approximates the volume on $M$ in the following sense: there exists a partition of $M$ into measurable subsets $\{V_i\}_{i=1}^N$ such that $V_i\subset B_{\varepsilon}(x_i)$ and $\vol(V_i)=\mu_i$ for every $i$. Denote this discrete metric-measure space by $\Gamma_{\ep}=(X_{\varepsilon},d|_{X_{\varepsilon}},\mu)$, and we write $\Gamma$ for short.
\end{definition}
Let $M,E,\nabla$ be defined as before, and we consider the $\rho$-connection Laplacian on this discrete metric-measure space $\Gamma$, acting on the restriction of the vector bundle $E$ onto $X_{\varepsilon}$. Namely, let $P=\{P_{x_i x_j}: d(x_i,x_j)<\rho\}$ be the parallel transport between points in $X_{\ep}$. We call $P$ a \emph{$\rho$-connection} on the restriction $E|_{X_{\varepsilon}}$. For $\bar{u}\in L^2(X_{\varepsilon},E|_{X_{\varepsilon}})$, we define
\begin{equation}\label{sec1-graph-connection}
\Delta^{\rho}_{\Gamma} \bar{u}(x_i):=\frac{2(n+2)}{\nu_n \rho^{n+2}} \sum_{d(x_i,x_j)<\rho} \mu_j \big(\bar{u}(x_i)-P_{x_i x_j}\bar{u}(x_j)\big).
\end{equation}
This operator is known as the \emph{graph connection Laplacian}. The graph connection Laplacian is a nonnegative self-adjoint operator of dimension $r(E)N$ with respect to the weighted discrete $L^2$-inner product, where $r(E)$ is the rank of the vector bundle $E$. We denote the $k$-th eigenvalue of $\Delta_{\Gamma}^{\rho}$ by $\til\lambda_k(\Gamma)$.

Our second result can be viewed as a discretized version of Theorem \ref{t:smoothcase}.

\begin{theorem}\label{t:graphcase}
Suppose that the absolute value of the sectional curvatures of $M$ 
and the norm of the curvature tensor of $\nabla$ are bounded by
constants $\KM$ and $\KE$. Then there exists $\rho_0=\rho_0(n,K_M,K_E)<r_{inj}(M)/2$,
such that for any $\rho<\rho_0$, \,$\varepsilon<\rho/4$, \,$k\leq r(E)N$ satisfying $\lambda_k<\rho^{-2}/16$, 
we have
$$
  \big| \til\la_k(\Gamma) - \la_k \big| \le
   C_{n,K_M}\big(\rho+\frac{\ep}{\rho}+\lambda_k^{1/2}\rho\big)\lambda_k +C_{n,K_E} (\rho+\frac{\ep}{\rho})\, .
$$
\end{theorem}

For compact Riemannian manifolds without boundary, Theorem \ref{t:smoothcase} and \ref{t:graphcase} imply the closeness between the $\rho$-connection Laplacian and the graph connection Laplacian in the spectral sense. In the case of trivial bundles, this gives another proof for the closeness between the $\rho$-Laplacian (on functions) and the graph Laplacian in the spectral sense (as a special case of Theorem 1.2 in \cite{BIK2}).

\medskip
This paper is organized as follows. We introduce the general concept of $\rho$-connection Laplacians for metric-measure spaces in Section \ref{sec:general}. In Section \ref{sec:smooth}, we focus on the case of smooth connections on Riemannian manifolds and prove Theorem \ref{t:smoothcase}. We turn to the graph connection Laplacian in Section \ref{sec:graph} and prove Theorem \ref{t:graphcase}. 

\medskip
\noindent
\textbf{Acknowledgement.} We are grateful to the anonymous referee for thoroughly reading our manuscript and making valuable suggestions for improvements.

\section{General metric-measure setup}
\label{sec:general}

Let $(X,d)$ be a metric space and $E$ be an (continuous) Euclidean or Hermitian vector bundle over $X$.
That is, we have a fiber bundle $\pi:E\to X$, and each fiber $E_x:=\pi^{-1}(x)$, $x\in X$,
is a real or complex vector space equipped 
with an Euclidean or Hermitian inner product $\langle\, ,\rangle_{E_x}$. When it is clear at which point the inner product is taken, we omit the subscript $E_x$.

Let $\rho>0$ be a small parameter. We denote by $X^2(\rho)$ the set of pairs
$(x,y)\in X\times X$ such that $d(x,y)\le\rho$.

\begin{definition}
A \emph{$\rho$-connection} on $E$ is a (Borel measurable)
family of uniformly bounded linear maps $P_{xy}:E_y\to E_x, \,\sup_{(x,y)}\|P_{xy}\| < \infty,$
where $(x,y)$ ranges over $X^2(\rho)$. That is, $P_{xy}$ transports vectors from
$E_y$ to $E_x$. 

A $\rho$-connection $P=\{P_{xy}\}$ is said to be {\em Euclidean}
(resp.\ {\em Hermitian})
if all maps $P_{xy}$ are Euclidean isometries
(resp.\ unitary operators).
A $\rho$-connection $P$ is said to be {\em symmetric},
if $P_{xy}$ is invertible and $P_{xy}^{-1}=P_{yx}$ for all $(x,y)\in X^2(\rho)$.
\end{definition}

Our primary example of $\rho$-connection is the one associated with a
connection $\nabla$ on an (smooth) Euclidean or Hermitian vector bundle
over a Riemannian manifold $M$.
Namely, there is a parallel transport canonically associated with the connection $\nabla$. 
Then given two points $x,y\in M$ with $d(x,y)\leq\rho<r_{inj}(M)$, 
one can transport vectors from $E_y$ to $E_x$ along the unique 
minimizing geodesic $[yx]$.

\smallskip
Now suppose $(X,d)$ is equipped with a measure $\mu$: we are working in a metric-measure space $(X,d,\mu)$. 
Assume that $X$ is compact and $\mu(X)<\infty$.
We denote by $L^2(X,E)$ the space of $L^2$-sections of
the vector bundle $E$.

The following expression $u(x) - P_{xy}(u(y))$,
where $u\in L^2(X,E)$ and $x,y\in X$, shows up frequently.
We introduce a short notation for it:
\be\label{e:defGa}
  \Gamma_{xy}(u) := u(x) - P_{xy}(u(y)) .
\ee
Note that $\Gamma_{xy}(u)$ is only defined for $(x,y)\in X^2(\rho)$
and it belongs to $E_x$.

Let $\alpha:X\to\R_+$ and $\beta:X^2(\rho)\to\R_+$
be positive $L^\infty$ functions satisfying that $\alpha$ is bounded away from 0
and $\beta$ is symmetric: $\beta(x,y)=\beta(y,x)$ for all $x,y$.
We define the {\em $\rho$-connection Laplacian} $\Delta^\rho_{\alpha,\beta}$
associated with the $\rho$-connection $P$ with weights $\alpha,\beta$
as follows. First we define an $L^2$-type inner product $\tprod{\, ,}_\alpha$ on $L^2(X,E)$
by
$$
  \tprod{u,v}_\alpha := \int_X \alpha(x) \, \big\langle u(x), v(x) \big\rangle_{E_x} \,d\mu(x),
$$ 
for $u,v\in L^2(X,E)$, 
and the associated norm $\tnorm{\cdot}_\alpha$ is
$$
  \tnorm{u}_\alpha^2 := \tprod{u,u}_{\alpha} = \int_X \alpha(x) \, |u(x)|^2 \,d\mu(x) .
$$
Here $|u(x)|^2=\langle u(x), u(x) \rangle$ is taken with respect to the
Euclidean or Hermitian inner product in the fiber $E_x$.
Note that the standard inner product on $L^2(X,E)$ corresponds
to the case $\alpha\equiv 1$, in which case the norm is denoted by the usual $\|u\|_{L^2}$.

Next we define a symmetric form $D_\beta^{\rho}$ on $L^2(X,E)$ by
\begin{equation}\label{def-energy}
  D_\beta^{\rho}(u,v) 
  = \frac12 \iint_{X^2(\rho)} \beta(x,y)\, \big\langle\Gamma_{xy}(u),\Gamma_{xy}(v)\big\rangle_{E_x} \, 
  d\mu(x) d\mu(y) .
\end{equation}
Finally, our Laplacian $\Delta^\rho_{\alpha,\beta}$ is the unique operator from $L^2(X,E)$ to itself
satisfying
\be\label{e:alpha-beta-self-adjoint}
  \tprod{\Delta^\rho_{\alpha,\beta}u, v}_\alpha = D_\beta^{\rho}(u,v)
\ee
for all $u,v\in L^2(X,E)$. 
In other words, $\Delta^\rho_{\alpha,\beta}$ is the self-adjoint 
operator on $L^2(X,\alpha\mu,E)$ associated with the
quadratic form $D_\beta^{\rho}$. Note that the boundedness of $P_{xy}$ 
and $\mu(X)$ imply that both $D_{\beta}^{\rho}$ and $\Delta^\rho_{\alpha,\beta}$
are bounded.

The following proposition gives an explicit formula for $\Delta^\rho_{\alpha,\beta}$.

\begin{proposition}\label{p:rholaplacian-formula}
Assume that $X$ is compact and $\mu(X)<\infty$. Then $\Delta^\rho_{\alpha,\beta}$ can be written as
\be\label{e:mm-laplace-general}
  \Delta^\rho_{\alpha,\beta}u(x)
  = \frac1{2\alpha(x)} \int_{B_\rho(x)} \beta(x,y)
  \big(\Gamma_{xy}(u) - P_{yx}^*\Gamma_{yx}(u)\big) \,d\mu(y) \, ,
\ee
where $P_{xy}^*:E_x\to E_y$ is the operator adjoint to $P_{xy}$ with respect to
the inner products on the fibers $E_x$ and $E_y$.

In particular, if the $\rho$-connection $P$ is Euclidean or Hermitian and is symmetric, then
\be\label{e:mm-laplace}
   \Delta^\rho_{\alpha,\beta}u(x)
  = \frac1{\alpha(x)} \int_{B_\rho(x)} \beta(x,y) \, \Gamma_{xy}(u) \,d\mu(y) \, .
\ee
\end{proposition}

\begin{proof}
The proof is a straightforward calculation.
For brevity, we write $dx$ and $dy$ instead of $d\mu(x)$ and $d\mu(y)$.
By definitions of $D_\beta^{\rho}$ and $\Gamma_{xy}(v)$, we have
$$
  D_\beta^{\rho}(u,v) 
  = \frac12 \iint_{X^2(\rho)} \beta(x,y) \big\langle \Gamma_{xy}(u), v(x) - P_{xy}(v(y)) \big\rangle \,dxdy .
$$
Expand it and rewrite the term $\langle \Gamma_{xy}(u), P_{xy}(v(y)) \rangle$ 
as follows:
$$
 \big\langle \Gamma_{xy}(u), P_{xy}(v(y)) \big\rangle_{E_x} 
 = \big\langle P_{xy}^*\Gamma_{xy}(u), v(y) \big\rangle_{E_y} .
$$
By swapping $x,y$ and using the symmetry of $\beta$, one gets
$$
 \iint_{X^2(\rho)} \beta(x,y) \big\langle P_{xy}^*\Gamma_{xy}(u), v(y) \big\rangle  \,dxdy
 = \iint_{X^2(\rho)} \beta(x,y) \big\langle P_{yx}^*\Gamma_{yx}(u), v(x) \big\rangle  \,dxdy .
$$
Substituting the last two formulae into the first one yields
$$
\begin{aligned}
  D_\beta^{\rho}(u,v) 
  &= \frac12\iint_{X^2(\rho)} \beta(x,y) \big\langle \Gamma_{xy}(u)-P_{yx}^*\Gamma_{yx}(u), v(x)) \big\rangle \,dxdy \\
  &= \frac12 \int_X \alpha(x) \bigg\langle\frac1{\alpha(x)}
   \int_{B_\rho(x)} \beta(x,y)  \big(\Gamma_{xy}(u)-P_{yx}^*\Gamma_{yx}(u)\big)\, dy , v(x)\bigg\rangle dx.
 \end{aligned}
$$
The right-hand side of the last formula is the $\tprod{\, ,}_\alpha$-product of
the right-hand side of \eqref{e:mm-laplace-general} and $v$. 
This proves \eqref{e:mm-laplace-general}.

To deduce \eqref{e:mm-laplace}, observe that $P_{yx}^*=P_{yx}^{-1}=P_{xy}$,
since $P$ is Euclidean or Hermitian and is symmetric.
Hence by (\ref{e:defGa}),
$$
  P_{yx}^*\G_{yx}(u) = P_{yx}^{-1}\big(u(y)-P_{yx}(u(x))\big)
  = P_{xy}(u(y)) - u(x) = - \G_{xy}(u) .
$$
This and \eqref{e:mm-laplace-general} prove \eqref{e:mm-laplace}.
\end{proof}

\begin{remark}
We can assume that $\beta(x,y)$ is defined for all pairs $(x,y)\in X^2(\rho)$ and
is equal to 0 whenever $d(x,y)>\rho$.
This allows us to assume that $P_{xy}$ is defined for all pairs $(x,y)\in X\times X$.
It does not matter how $P$ is extended to pairs $x,y$ with $d(x,y)>\rho$,
since in this case it will always be multiplied by 0.
This allows us to write integration over $X$ rather than over $\rho$-balls
whenever convenient.
\end{remark}

Denote by $\sigma( \Delta^\rho_{\alpha,\beta})$ the spectrum of $ \Delta^\rho_{\alpha,\beta}$ and by
$\sigma_{ess}( \Delta^\rho_{\alpha,\beta})$ the essential spectrum.

\begin{corollary} \label{c:essential}
$$
\sigma_{ess}( \Delta^\rho_{\alpha,\beta}) \subset \big[a(\rho, \alpha, \beta),\, \infty \big),
$$
where $a(\rho, \alpha, \beta)=\min_{x  \in X} \frac{1}{2 \alpha(x)}
\int_X \beta(x, y) dy$.

In particular, if the $\rho$-connection $P$ is Euclidean or Hermitian and is symmetric, then
$$
\sigma_{ess}( \Delta^\rho_{\alpha,\beta}) \subset \big[2 a(\rho, \alpha, \beta),\, \infty \big).
$$
\end{corollary}
\begin{proof}
Due to \eqref{e:defGa} and \eqref{e:mm-laplace-general}, $\Delta^\rho_{\alpha,\beta}$ is the sum of three operators: the multiplication operator
$$
\Delta_1 u(x) = \frac{1}{2 \alpha(x)} u(x)
\int_X \beta(x, y)\,dy,
$$
the operator
$$
\Delta_2 u(x)=\frac{1}{2 \alpha(x)}
\int_X \beta(x, y) P^*_{yx} P_{yx} u(x)\, dy, 
$$
and
$$
\Delta_3 u(x)=\frac{1}{2 \alpha(x)}
\int_X \beta(x, y) \left(-P_{xy}-P^*_{yx} \right) u(y) \,dy.
$$
Observe that $\sigma(\Delta_1) \subset [a(\rho, \alpha, \beta), \infty)$;
while $\Delta_2$, as seen from its quadratic form
$$
 \tprod{\Delta_2 u, u}_\alpha= \frac12 \iint_{X^2(\rho)} \beta(x, y) \langle P_{xy} u,\, 
 P_{xy} u \rangle dxdy,
$$
is non-negative, and $ \Delta_3$ is compact as an operator with bounded kernel, which proves the first claim. The second claim follows from the same considerations by using the simpler form \eqref{e:mm-laplace}.
\end{proof}

\subsection*{Examples}
In this paper, we only need a few choices for $\alpha$ and $\beta$.
First observe that the Riemannian $\rho$-connection Laplacian (\ref{sec1-def-rhoLaplacian}) is obtained
by using $\alpha(x)=1$ and
$\beta(x,y)=\frac{2(n+2)}{\nu_n\rho^{n+2}}$ if $d(x,y)\le\rho$.
Equivalently, one can use $\alpha(x)=\frac{\nu_n\rho^{n+2}}{2(n+2)}$
and $\beta(x,y)=1$  if $d(x,y)\le\rho$. Another convenient normalization, as seen in \cite{BIK2}, is by volumes of $\rho$-balls:
$\beta(x,y)=1$ if $d(x,y)\le\rho$ and 
$\alpha(x)=\rho^2\mu(B_\rho(x))$ for all $x\in X$.
We call the operator $\Delta^\rho_{\alpha,\beta}$
with these $\alpha,\beta$ the
{\em volume-normalized} $\rho$-connection Laplacian.

Note that Laplacians on real-or complex-valued functions are a special
case of connection Laplacians. Namely, for a Riemannian
manifold $M$, one simply considers the trivial bundle $E=M\times\R$ or $E=M\times\C$
equipped with the trivial connection $\nabla$. The sections of the trivial bundle
are functions on $M$ and the connection Laplacian
is simply the Laplace-Beltrami operator (on functions). Similarly, for a metric-measure space $X$, 
one can consider the same trivial bundle
with the trivial $\rho$-connection defined by $P_{xy}(y,t)=(x,t)$
for $x,y\in X$ and $t\in\R$ (or $t\in\C$).
Then \eqref{e:mm-laplace} boils down to
$$
 \Delta^\rho_{\alpha,\beta}u(x) = \frac1{\alpha(x)} \int_{B_\rho(x)} \beta(x,y) \big(u(x)-u(y) \big)\,d\mu(y),
$$
where $u\in L^2(X)$.
Such operators are called $\rho$-Laplacians in \cite{BIK2}.
Some analogues of the results of this paper
in the case of Laplacians on functions
can be found in \cite{BIK1,BIK2,LM,L}.

\subsection*{Spectra of $\rho$-connection Laplacians}

Since $\Delta^\rho_{\alpha,\beta}$ is self-adjoint with respect to
the $L^2$-compatible inner product $\tprod{\, ,}_\alpha$
and the corresponding quadratic form $D_\beta^{\rho}$
is positive semi-definite,
the spectrum of $\Delta^\rho_{\alpha,\beta}$ is contained in $\R_{\geq 0}$.

This spectrum consists of the discrete and essential spectra.
It follows from Corollary \ref{c:essential} that the essential spectrum
has a lower bound in all cases in question.
We are only interested in the part of the spectrum below this bound.
We enumerate this part of the spectrum
as follows (cf.\ Notation 2.1 in \cite{BIK2}).

\begin{notation}
\label{n:lambda}
Denote by $\til\la_\infty=\til\la_\infty(E,P,\rho,\alpha,\beta)$ 
the infimum of the essential spectrum of $\Delta^\rho_{\alpha,\beta}$.
If the essential spectrum is empty
(e.g. if $X$ is a discrete space),
we set $\til\la_\infty=\infty$.
For every $k\in\N_+$, we define $\til\la_k=\til\la_k(E,P,\rho,\alpha,\beta)\in[0,+\infty]$ as follows.
Let $0\le \til\la_1\le \til\la_2\le\cdots$ be the eigenvalues of $\Delta^\rho_{\alpha,\beta}$
(counting multiplicities)
that are smaller than $\til\la_\infty$. 
If there are only finitely many of such eigenvalues, 
we set
$\til\la_k=\til\la_\infty$ for all larger values of~$k$.

We abuse the language and refer to $\til\la_k(E,P,\rho,\alpha,\beta)$ as
the \emph{$k$-th eigenvalue} of $\Delta^\rho_{\alpha,\beta}$ even though
it may be equal to $\til\la_\infty$.
\end{notation}

By the standard min-max formula, for every $k\in\N_+$, we have
\be \label{e:minmax}
\til\la_k(E,P,\rho,\alpha,\beta) = \inf\nolimits_{\strut\dim L=k} \ \sup\nolimits_{\strut u \in L\setminus\{0\}}
\left(\frac{ D_\beta^{\rho}(u,u)}{\tnorm{u}_\alpha^2}\right),
\ee
where the infimum is taken over all $k$-dimensional
linear subspaces $L$ of $L^2(X,E)$.
We emphasize that \eqref{e:minmax}
holds in both cases of $\til\la_k<\til\la_\infty$ and $\til\la_k=\til\la_\infty$.

\section{Smooth connections on Riemannian manifolds}
\label{sec:smooth}

In this section, suppose $X=M^n$ is
a compact Riemannian manifold of dimension $n$ without boundary, and $E$ 
is a smooth Euclidean (or Hermitian) vector bundle over $M$
equipped with a smooth Euclidean (or Hermitian) connection $\nabla$.
Recall that an Euclidean (resp. Hermitian) connection is a connection that is compatible 
with the Euclidean (resp. Hermitian) metric on the vector bundle.
For a (sufficiently smooth) section $u$ of $E$,
$\nabla u$ is a section of the fiber bundle $\Hom(TM,E)$ over $M$.
Here $\Hom(TM,E)$ is the fiber bundle over $M$ whose fiber over $x\in M$
is the space $\Hom(T_xM,E_x)$ of $\R$-linear maps from $T_xM$ to $E_x$.
Note that in the case when $E$ is a complex fiber bundle,
$\Hom(T_xM,E_x)$ has a natural complex structure.
The standard (Euclidean or Hermitian) inner product on $\Hom(T_xM,E_x)$ is defined by
$$
  \langle \xi,\eta\rangle = \sum_{i=1}^n \langle \xi(e_i), \eta(e_i) \rangle_{E_x}
$$
for $\xi,\eta\in \Hom(T_xM,E_x)$,
where $\{e_i\}$ is an orthonormal basis of $T_xM$.
This defines the standard norm
\be\label{e:norm-nabla}
  \|\xi\|^2 = \sum_{i=1}^n |\xi(e_i)|^2
\ee
on fibers and the standard
$L^2$-norm on sections of $\Hom(TM,E)$.

Let $\Delta=\nabla^*\nabla$ be the connection Laplacian of $\nabla$ (e.g. \cite[Chapter 7.3.2]{PP} or \cite{KOP}).
It is a nonnegative self-adjoint operator
acting on $H^2$-sections of $E$. The corresponding energy functional
is given by
$$
  \langle \Delta u,u \rangle_{L^2} = \|\nabla u\|_{L^2}^2
$$
for $u\in H^2(M,E)$.

Since $\Delta$ is a nonnegative self-adjoint elliptic operator,
it has a discrete spectrum $0\le\lambda_1\le\lambda_2\le\dots$,
and $\lambda_k\to\infty$ as $k\to\infty$.
The min-max formula for $\lambda_k$ takes the form
\be \label{e:minmax-smooth}
\la_k = \inf\nolimits_{\strut\dim L=k} \ \sup\nolimits_{\strut u \in L\setminus\{0\}}
\left(\frac{\|\nabla u\|^2_{L^2}}{\|u\|^2_{L^2}}\right)
\ee
where the infimum is taken over all $k$-dimensional
linear subspaces $L$ of $H^1(M,E)$.
The goal of this section is to prove that $\la_k$
are approximated by eigenvalues of a $\rho$-connection
Laplacian defined below.

\smallskip
Let $\rho>0$ be smaller than the injectivity radius of $M$.
Let $P=\{P_{xy}\}$ be the $\rho$-connection associated with the connection $\nabla$,
which is given by the parallel transport from $y$ to $x$ along the unique minimizing geodesic $[yx]$.
This particular $\rho$-connection $P$ is unitary and symmetric since the connection $\nabla$ is Euclidean (or Hermitian).
We consider the $\rho$-connection Laplacian
$\Delta^\rho=\Delta^\rho_{\alpha,\beta}$ given by \eqref{e:mm-laplace}
with $\alpha(x)=1$ and
$\beta(x,y)=\frac{2(n+2)}{\nu_n\rho^{n+2}}$ if $d(x,y)\le\rho$.
That is,
\be\label{e:mm-laplace-riem}
   \Delta^\rho u(x)
  = \frac{2(n+2)}{\nu_n\rho^{n+2}} \int_{B_\rho(x)} \G_{xy}(u) \,dy.
\ee
Recall that $\Gamma_{xy}(u)$ is defined in \eqref{e:defGa}.
Here and later on in this paper, we denote by $dx,dy$ the
integration with respect to the Riemannian volume on $M$.
Denote by $\til\la_k$ the $k$-th eigenvalue
of $\Delta^\rho$ (see Notation \ref{n:lambda}).

We introduce a quadratic form $D^{\rho}$ on $L^2(M,E)$ by
\be\label{e:defD}
  D^{\rho}(u) = \int_M\int_{B_\rho (x)}  |\G_{xy}(u)|^2 \, dx dy .
\ee
Note that for the constant weight $\beta(x,y)=\frac{2(n+2)}{\nu_n\rho^{n+2}}$
chosen for $\Delta^\rho$, we have
$$
  D_\beta^{\rho}(u,u) =  \frac{n+2}{\nu_n\rho^{n+2}} D^{\rho}(u) .
$$
Hence \eqref{e:minmax} takes the form
\be \label{e:minmax-tilde}
\til\la_k = \frac{n+2}{\nu_n\rho^{n+2}} \
 \inf\nolimits_{\strut\dim L=k} \ \sup\nolimits_{\strut u \in L\setminus\{0\}}
\left(\frac{D^{\rho}(u)}{\|u\|_{L^2}^2}\right) .
\ee

\smallskip
The rest of this section is a proof of Theorem \ref{t:smoothcase}.

\subsection{Preparations and notations}


For $x\in M$, denote by $\exp_x:T_xM\to M$ the Riemannian
exponential map. We only need its restriction onto
the $\rho$-ball $\mathcal{B}_\rho(0)\subset T_xM$.
For $v\in T_xM$, denote by $J_x(v)$ the Jacobian of $\exp_x$ at $v$.
Let $J_{\min}(r)$ and $J_{\max}(r)$ denote the minimum and maximum
of $J_x(v)$ over all $x,v$ with $|v|\le r$.
The Rauch Comparison Theorem implies that
$$
 \left(\frac{\sin r}{r}\right)^{n-1}
  \le J_{\min}(r) \le 1 \le J_{\max}(r) 
  \le \left(\frac{\sinh r}{r}\right)^{n-1},
$$
for $r<K_M^{1/2}\rho<c_n$, see \eqref{e:rho-bound}.
In particular,
\be\label{e:rauch}
  (1+C_n\KM\rho^2)^{-1} 
  \le J_{\min}(r) \le 1 \le J_{\max}(r) 
  \le 1+C_n\KM\rho^2 .
\ee
Moreover, we choose $c_n$ to be sufficiently small such that $J_{\max}(r)/J_{\min}(r)<2$. Later we will mostly take $r=\rho$ and we denote $J_{\min}:=J_{\min}(\rho),\, J_{\max}:=J_{\max}(\rho)$ for short.

As a consequence of \eqref{e:rauch}, Corollary \ref{c:essential} implies that
\begin{equation}
\til\lambda_\infty\geq \frac{2(n+2)}{\nu_n \rho^{n+2}}\, \min_{x \in M} \mu(B_\rho(x))
\geq \frac{2(n+2)}{ \rho^{2}} \frac{1}{1+C_n K_M \rho^2} \geq \sigma_n \rho^{-2},
\end{equation}
for some constant $\sigma_n$ depending only on $n$ due to our choice of $\rho$ in \eqref{e:rho-bound}.
This shows that $\til\lambda_\infty$ is of order $\rho^{-2}$ in the present case. 

\smallskip
Later we use the following well-known inequality. We did not find a precise
reference for it, so we give a short proof here.

\begin{lemma} \label{l:parallel}
Let $\ga_s\co [0,1]\to M$ be a smooth family of paths 
from a fixed point $y\in M$ to $x(s)$, $s \in [-\ep,\ep]$.
 Let $P_{\ga_s} (v)$ be the $\nabla$-parallel transport along $\ga_s$ 
 of $v \in E_{y} $ to $E_{x(s)}$.
Then
\be \label{e:parallel_estimate}
\big| \nabla_{s} P_{\ga_s}(v) \big| \leq |v| \cdot K_E \cdot\text{\rm length}(\ga_s) \cdot
 \sup_{t\in[0,1]} \left|\frac {d \ga_s(t)}{ds} \right|.
\ee
\end{lemma}

\begin{proof}
Let
$
v(t, s)= P_{\ga_s[0, t]}(v),
$
where $P_{\ga_s[t_1, t_2]}$ is the $\nabla$-parallel transport along $\ga_s$ from $\ga_s(t_1)$ to 
$\ga_s(t_2)$. Note that $P_{\ga_s[t_1, t_2]}$ is a unitary operator.

Observe that $\nabla_{t} \,v(t, s)=0$ and thus
$\nabla_{s} \nabla_{t}\, v(t,s)=0$.
Hence, using the definition of the curvature operator $R_E$, we see that
\be \label{e:variation}
\nabla_{t} \nabla_{s} v(t,s)= 
R_E\left(\frac{d\ga_s(t)}{dt},\,\frac{d\ga_s(t)}{ds}\right) v(t, s)
\in E_{\ga_s(t)}.
\ee
Estimating the right-hand side yields
\be \label{e:variation-bound}
 |\nabla_{t} \nabla_{s} v(t,s)| \le |v| \cdot \KE \cdot \left| \frac{d\ga_s(t)}{dt} \right| \cdot
 \sup_{t\in[0,1]} \left|\frac {d \ga_s(t)}{ds} \right|,
\ee
where we have used the fact that $|v(t,s)|=|v|$ since the parallel transport $P$ is unitary.
The plan is to integrate \eqref{e:variation} with respect to $t$. However, the fibers $E_{\ga_s(t)}$
vary with $t$. Thus, we use $P_{\ga_s[t, 1]}$ to identify them with $E_{x(s)}$.
Recall that by the definition of parallel translations, for any vector field $X$ along $\ga_s$, one has
$$
  \frac d{dt} \bigl(P_{\ga_s[t,1]} X(t)\bigr) 
  = P_{\ga_s[t,1]} \big( \nabla_t X(t) \big) .
$$
Note that the vectors under $\frac d{dt}$ in this formula lie in the same vector space $E_{x(s)}$
for all $t$. We apply this to $X(t)= \nabla_{s} v(t,s)$ and obtain
$$
 \frac d{dt} \bigl(P_{\ga_s[t,1]}  \nabla_{s} v(t,s)\bigr)
  =P_{\ga_s[t,1]} \big(\nabla_{t} \nabla_{s} v(t,s)\big) .
$$
Integrating with respect to $t$ and taking into account that $\nabla_s v(0,s)=0$ yield
$$
 \nabla_{s} P_{\ga_s}(v) =
  \nabla_{s} v(1,s) = \int_0^1 P_{\ga_s[t,1]} \big(\nabla_{t} \nabla_{s} v(t,s)\big) \,dt .
$$
Therefore, using the fact that $P_{\ga_s[t,1]}$ is unitary, we have
$$
 |\nabla_{s} P_{\ga_s}(v)| \le
  \int_0^1 \bigl|\nabla_{t} \nabla_{s} v(t,s)\bigr| \,dt 
  \le |v| \cdot \KE \cdot
 \sup_{t\in[0,1]} \left|\frac {d \ga_s(t)}{ds} \right|
 \cdot \int_0^1 \left| \frac{d\ga_s(t)}{dt} \right| \,dt ,
$$
where the second inequality follows from \eqref{e:variation-bound}.
This formula is exactly \eqref{e:parallel_estimate}.
\end{proof}

We need the following elementary fact from the linear algebra (e.g. \cite[\S 2.3]{BIK1}): If $S$ is a quadratic form on $\R^n$,
then
\be\label{e:Qoverball}
  \int_{B_\rho(0)} S(x) \,dx = \frac{\nu_n\rho^{n+2}}{n+2} \tr(S) .
\ee

In the following two subsections, we control the upper and lower bounds for $\til\la_k$ by following the method we established in \cite{BIK1}.

\subsection{Upper bound for $\widetilde\lambda_k$}

\begin{lemma} \label{l:Du-upper}
For any $u\in H^1(M,E)$, we have
$$
  D^{\rho}(u) \le J_{\max} \frac{\nu_n}{n+2} \, \rho^{n+2} \, \|\nabla u\|^2_{L^2}.
$$
\end{lemma}

\begin{proof}
This lemma is similar to Lemma 3.3 in \cite{BIK1} and the proof
is essentially the same.
We may assume that $u$ is smooth.
By substituting $y=\exp_x(v)$, we have
\begin{eqnarray*}
  \int_{B_\rho(x)} |\G_{xy}(u)|^2 \,dy
  &=& \int_{\mathcal{B}_\rho(0)\subset T_xM}  |\G_{x,\exp_x(v)}(u)|^2 J_x(v)\,dv \\
  &\le& J_{\max} \int_{\mathcal{B}_\rho(0)}  |\G_{x,\exp_x(v)}(u)|^2 \,dv .
\end{eqnarray*}
Hence
\be\label{e:DJmaxA}
  D^{\rho}(u) \le J_{\max} A,
\ee
where
$$
  A =  \int_M \int_{\mathcal{B}_\rho(0)\subset T_xM}  |\G_{x,\exp_x(v)}(u)|^2 \,dv dx .
$$
Note that the right-hand side is an integral with respect to the
Liouville measure on $TM$.
Let us estimate $A$.

For every constant-speed minimizing geodesic $\ga:[0,1]\to M$, we have
$$
  \G_{\ga(0)\ga(1)} = u(\ga(0)) - P_{\ga(0)\ga(1)} \big(u(\gamma(1))\big) 
  = - \int_0^1 \frac d{dt}  P_{\ga(0)\ga(t)} \big( u(\gamma(t)) \big) \, dt,
$$
and
$$
 \frac d{dt}  P_{\ga(0)\ga(t)} \big(u(\gamma(t)) \big) = P_{\ga(0)\ga(t)} \big( \nabla_{\dot\ga(t)} u \big)
$$
by the definition of the parallel transport $P$. Therefore,
$$
  |\G_{\ga(0)\ga(1)}| \le \int_0^1 \bigl| P_{\ga(0)\ga(t)} \big( \nabla_{\dot\ga(t)} u \big) \bigr|
    = \int_0^1 |\nabla_{\dot\ga(t)} u| \,dt,
$$
where the last equality is due to $P$ being unitary. Then,
\be\label{e:nabla-ineq}
  |\G_{\ga(0)\ga(1)}|^2 \le \int_0^1 |\nabla_{\dot\ga(t)} u|^2 \,dt .
\ee
For $x\in M$ and $v\in \mathcal{B}_\rho(0)\subset T_xM$, denote by $\ga_{x,v}$ the constant-speed geodesic
with the initial data $\ga_{x,v}(0)=x$ and $\dot\ga_{x,v}(0)=v$.
Equivalently, $\ga_{x,v}(t)=\exp_x(tv)$.
Applying \eqref{e:nabla-ineq} to $\ga_{x,v}$ yields
$$
  |\G_{x,\exp_x(v)}(u)|^2 \le \int_0^1 |\nabla_{\dot\ga_{x,v}(t)} u|^2 \,dt .
$$
This and the definition of $A$ imply that
$$
  A \le \int_0^1 f(t)\,dt,
$$
where
$$
  f(t)=\int_M \int_{\mathcal{B}_\rho(0)\subset T_xM}  |\nabla_{\dot\ga_{x,v}(t)}u|^2 \,dv dx .
$$
Note that $\dot\ga_{x,v}(t)$ is the image of $v$ under the time $t$ map of the geodesic flow.
Since the geodesic flow preserves the Liouville measure and the subset $\{(x,v)\in TM: v\in \mathcal{B}_{\rho}(0)\subset T_x M\}$, $f(t)$ does not depend on $t$.
Hence,
$$
  A \le f(0) = \int_M \int_{\mathcal{B}_\rho(0)\subset T_x M}  |\nabla_{v}u|^2 \,dv dx 
  = \int_M \frac{\nu_n\rho^{n+2}}{n+2}  \|\nabla u(x)\|^2 \, dx
  = \frac{\nu_n\rho^{n+2}}{n+2}  \|\nabla u\|^2_{L^2},
$$
where the second equality follows from \eqref{e:Qoverball}.
This and \eqref{e:DJmaxA} yield the lemma.
\end{proof}

The lemma above gives an upper bound for $\til\la_k$.

\begin{proposition}
\label{p:wtla-upper-bound}
For every $k\in\N_+$, we have
$$
  \til\la_k \le J_{\max} \la_k \le (1+C_n\KM\rho^2) \la_k .
$$
\end{proposition}

\begin{proof}
The second inequality follows from \eqref{e:rauch}.
The first one follows immediately from combining
the min-max formulae 
\eqref{e:minmax-smooth} and \eqref{e:minmax-tilde}
and Lemma \ref{l:Du-upper}.
\end{proof}

\smallskip
\subsection{Lower bound for $\widetilde\lambda_k$}

As in \cite[Section 5]{BIK1}, define $\psi:\R_{\geq 0}\to\R_{\geq 0}$ by
$$
  \psi(t) = \begin{cases}
  \frac{n+2}{2\nu_n} (1-t^2), & 0\le t\le 1, \\
  0, & t\ge 1 .
  \end{cases}
$$
The normalization constant $\frac{n+2}{2\nu_n}$ is chosen so that
$\int_{\R^n} \psi(|x|)\,dx=1$.

We define $k_{\rho}:M\times M\to\R_{\geq 0}$ by
\be\label{e:smoothing-beta}
  k_{\rho}(x,y) = \rho^{-n} \psi \big(\frac{d(x,y)}{\rho} \big) ,
\ee
and $\theta:M\to\R_{\geq 0}$ by
\be\label{e:smoothing-alpha}
  \theta(x) = \int_M k_{\rho}(x,y) \,dy = \int_{B_\rho(x)} k_{\rho}(x,y) \,dy .
\ee
(The second identity follows from the fact that $k_{\rho}(x,y)=0$
if $d(x,y)\ge\rho$.)


We need the following estimates on $\theta$:
\be\label{e:theta-bound}
  J_{\min} \le \theta(x) \le J_{\max}
\ee
and
\be\label{e:dtheta-bound}
  |d_x \theta| \le C_n\KM \rho
\ee
for all $x\in M$. See \cite[Lemma 5.1]{BIK1} for a proof.

Define a convolution operator $I:L^2(M,E)\to C^{0,1}(M,E)$ by
$$
  Iu(x) 
  = \frac 1{\theta(x)}\int_M k_{\rho}(x,y) P_{xy}(u(y)) \,dy
  = \frac 1{\theta(x)}\int_{B_\rho(x)} k_{\rho}(x,y) P_{xy}(u(y)) \,dy,
$$
for $x\in M$ (compare with \cite[Definition 5.2]{BIK1}). We estimate the energy and the $L^2$-norm of $Iu$ in the following two lemmas.

\begin{lemma}
\label{l:Iu-lower-bound}
For any $u\in L^2(M,E)$, we have
$$
  \|Iu\|^2_{L^2} \ge \frac{J_{\min}}{J_{\max}}\, \|u\|^2_{L^2} - \frac{n+2}{2\nu_n\rho^n} D^{\rho}(u) .
$$
\end{lemma}

\begin{proof}
Consider the weighted $\rho$-connection Laplacian $\Delta^\rho_{\theta,k_{\rho}}$ defined in Section \ref{sec:general}.
By \eqref{e:mm-laplace},
$$
 \Delta^\rho_{\theta,k_{\rho}} u(x) 
 = \frac {1}{\theta(x)}\int_{B_\rho(x)} k_{\rho}(x,y)\bigl(u(x)-P_{xy}(u(y))\bigr) \,dy 
 = u(x)-Iu(x) .
$$
where the second equality follows from the definition \eqref{e:smoothing-alpha}.
Equivalently, $Iu=u-\Delta^\rho_{\theta,k_{\rho}} u$. Therefore by \eqref{e:alpha-beta-self-adjoint},
$$
  \tnorm{Iu}_\theta^2
  = \tnorm{u}_\theta^2 - 2\tprod{\Delta^\rho_{\theta,k_{\rho}} u, u}_\theta
   + \tnorm{\Delta^\rho_{\theta,k_{\rho}} u}_\theta^2
  \ge \tnorm{u}_\theta^2 - 2\tprod{\Delta^\rho_{\theta,k_{\rho}} u, u}_\theta
  = \tnorm{u}_\theta^2 - 2 D_{k_{\rho}}^{\rho}(u,u).
$$
Observe that
$$
  D_{k_{\rho}}^{\rho}(u,u) \le \frac12 \max_{x,y} k_{\rho}(x,y) \cdot D^{\rho}(u) \le \frac{n+2}{4\nu_n\rho^n} D^{\rho}(u) .
$$
Thus
$$
   \tnorm{Iu}_\theta^2 \ge  \tnorm{u}_\theta^2 - \frac{n+2}{2\nu_n\rho^n} D^{\rho}(u) .
$$
Then the lemma follows from the inequality above, $J_{\max}\ge 1$, and the following trivial estimates:
$$
 \tnorm{Iu}_\theta^2 \le \max_x\theta(x) \cdot \|Iu\|^2_{L^2} \le J_{\max} \|Iu\|^2_{L^2} ,
$$
and 
$$
 \tnorm{u}_\theta^2 \ge \min_x\theta(x) \cdot \|u\|^2_{L^2} \ge J_{\min} \|u\|^2_{L^2} ,
$$
due to \eqref{e:theta-bound}.
\end{proof}

\begin{lemma}
\label{l:nablaIu-upper-bound}
For any $u\in L^2(M,E)$, we have
$$
  \|\nabla Iu\|_{L^2} \le \left(1+C_n\KM\rho^2\right) \sqrt{\frac{n+2}{\nu_n\rho^{n+2}} D^{\rho}(u)} 
  + C_nK_E\,\rho \,\|u\|_{L^2} .
$$
\end{lemma}

\begin{proof}
We rewrite the definition of $I$ as
\be\label{e:I1}
 Iu(x) = \int_{M} \wt{k_{\rho}}(x,y) P_{xy}(u(y)) \,dy,
\ee
where
$$
 \wt{k_{\rho}}(x,y) = \theta(x)^{-1} k_{\rho}(x,y) .
$$
Note that for any $x\in M$,
$$ 
 \int_M \wt{k_{\rho}}(x,y)\,dy = 1 .
$$
Differentiating this identity, we get
\be\label{e:dtildebeta}
 \int_M d_x\wt{k_{\rho}}(x,y)\,dy = 0,
\ee
where $d_x$ denotes the differential with respect to~$x$.
Differentiating \eqref{e:I1} yields
$$
 \nabla Iu(x) = A_0(x)+A_3(x),
$$
where
$$
 A_0(x) = \int_M d_x\wt{k_{\rho}}(x,y)\otimes P_{xy}(u(y))\,dy
 = \int_{B_\rho(x)} d_x\wt{k_{\rho}}(x,y)\otimes P_{xy}(u(y))\,dy,
$$
and
$$
  A_3(x) = \int_M\wt{k_{\rho}}(x,y) \, \nabla V_{y,u}(x) \,dy
  = \int_{B_\rho(x)}\wt{k_{\rho}}(x,y) \, \nabla V_{y,u}(x) \,dy.
$$
In the above, $V_{y,u}(\cdot)$ is the section of $E|_{B_\rho(y)}$ defined by
\begin{equation}\label{def-Vyu}
 V_{y,u}(z) = P_{zy}(u(y)) .
\end{equation}
For better understanding, observe that $\nabla V_{y,u}=0$ if $\nabla$ is a flat connection.
In general, we have an estimate
$$
  \|\nabla V_{y,u}(x)\| \le C_n \KE\, d(x,y)\, {|u(y)|} \le C_n \KE\, \rho\, {|u(y)|}  ,
$$
where the norm at the left-hand side is defined by \eqref{e:norm-nabla}.
Indeed, for any unit vector $w\in T_xM$, we have
\begin{equation}
\label{error}
  |\nabla_w V_{y,u}(x)| \le  \KE\, d(x,y)\, {|u(y)|}\, |w| =\KE\, d(x,y)\, {|u(y)|} .
\end{equation}
This follows from Lemma \ref{l:parallel} applied to $u(y)$ in place of $v$, and applied to
a family of minimizing
geodesics from $y$ to points $x(s)$ where $x(0)=x$ and $\dot x(0)=w$.
Note that due to the curvature bound $\textrm{Sec}_M\leq K_M$, the Rauch comparison theorem implies that $\sup_t\big|\frac {d \ga_s(t)}{ds} \big|$ in 
\eqref{e:parallel_estimate} is attained at $t=1$ when $\rho$ satisfies \eqref{e:rho-bound}, see e.g. \cite[Chapter 4, Corollary 2.8(1)]{S}.

Hence by the Cauchy-Schwarz inequality,
$$
 \|A_3(x)\| \le C_n\KE\rho\int_{B_\rho(x)} \wt{k_{\rho}}(x,y) |u(y)| \, dy
 \le C_n\KE\, \rho^{1-\frac{n}{2}} \left(\int_{B_\rho(x)} |u(y)|^2 \, dy\right)^{\frac12},
$$
where we used the estimates $\wt{k_{\rho}}(x,y)\le C_n\rho^{-n}$ and
$\vol(B_\rho(x))\le C_n\rho^n$ due to \eqref{e:rauch}. Thus
\be\label{e:A3bound}
 \|A_3\|_{L^2} = \left(\int_M \|A_3(x)\|^2\,dx\right)^{\frac12} \le C_nK_E\,\rho \,\|u\|_{L^2}.
\ee

Next we turn to $A_0$. Using \eqref{e:dtildebeta},
$$
 A_0(x)= \int_{B_\rho(x)} d_x\wt{k_{\rho}}(x,y)\otimes \bigl(P_{xy}(u(y))-u(x)\bigr)\,dy
 = -\int_{B_\rho(x)} d_x\wt{k_{\rho}}(x,y)\otimes \G_{xy}(u)\,dy .
$$
Since $d_x\wt{k_{\rho}}(x,y) = \theta(x)^{-1} d_x k_{\rho}(x,y) - \theta(x)^{-2} d_x\theta(x) \cdot k_{\rho}(x,y)$, we split $A_0$ into two terms:
$$
 A_0(x) = A_1(x) + A_2(x),
$$
where
$$
 A_1(x) = -\theta(x)^{-1}  \int_{B_\rho(x)} d_x k_{\rho}(x,y)\otimes \G_{xy}(u)\,dy,
$$
and
$$
 A_2(x) = \theta(x)^{-2} d_x\theta(x) \otimes \int_{B_\rho(x)} k_{\rho}(x,y) \G_{xy}(u) \,dy.
$$
In view of \eqref{e:rauch}, \eqref{e:theta-bound} and \eqref{e:dtheta-bound}, 
$$
  \|A_2(x)\| \le C_n\KM\rho^{1-\frac{n}{2}} \left(\int_{B_\rho(x)} |\G_{xy}(u)|^2\,dy\right)^{\frac12}.
$$
This inequality and the definition \eqref{e:defD} of $D^{\rho}(u)$ yield that
\be\label{e:A2bound}
 \|A_2\|_{L^2} \le C_n\KM \, \rho^{1-\frac{n}{2}} \sqrt{D^{\rho}(u)} = C_n\KM\, \rho^2 \, \sqrt {\frac1{\rho^{n+2}} D^{\rho}(u)}\, .
\ee

Now we estimate $A_1$. From the definition of $k_{\rho}$, for $w\in T_xM$, we have
$$
 d_x k_{\rho}(x,y)\cdot w = \frac{n+2}{\nu_n\rho^{n+2}}\, \langle \exp_x^{-1}(y), w\rangle
$$
where $\langle\, ,\rangle$ is the Riemannian inner product in $T_xM$.
Substituting this into the formula for $A_1$ yields
$$
 A_1(x)\cdot w = -\theta(x)^{-1} \frac{n+2}{\nu_n\rho^{n+2}} \int_{B_\rho(x)}
  \langle \exp_x^{-1}(y), w\rangle \, \G_{xy}(u)\,dy .
$$
Using the substitution $y=\exp_x(v)$, we get
\be\label{e:A1a}
 A_1(x)\cdot w = -\theta(x)^{-1} \frac{n+2}{\nu_n\rho^{n+2}} 
 \int_{\mathcal{B}_\rho(0)\subset T_xM} \langle v,w\rangle\, \phi(v) J_x(v)\, dv,
\ee
where
$$
 \phi(v) = \G_{x,\exp_x(v)}(u) \in E_x .
$$

To proceed we need the following sublemma.
\begin{sublemma}
For any $L^2$ function $f:\R^n\to\R$, one has
$$
 \bigg| \int_{\mathcal{B}_\rho(0)\subset\R^n} f(v)\, v \,dv \bigg|^2 
 \le \frac{\nu_n\rho^{n+2}}{n+2} \int_{\mathcal{B}_\rho(0)} f(v)^2 \,dv .
$$
\end{sublemma}

\begin{proof}
Denote
$$
 F = \int_{\mathcal{B}_\rho(0)\subset\R^n} f(v)\, v \,dv \in\R^n.
$$
Let $v_0\in\R^n$ be the unit vector in the direction of $F$.
Then
$$
  |F| = \langle F,v_0\rangle = \int_{\mathcal{B}_\rho(0)} f(v) \langle v,v_0\rangle \, dv .
$$
Therefore 
$$
 |F|^2 \le \Big(\int_{\mathcal{B}_\rho(0)} f(v)^2\, dv\Big)
 \Big( \int_{\mathcal{B}_\rho(0)} \langle v,v_0\rangle^2\, dv \Big) .
$$
Using \eqref{e:Qoverball}, the second integral equals to $\frac{\nu_n\rho^{n+2}}{n+2}$.
The sublemma follows.
\end{proof}

We use the sublemma and \eqref{e:A1a} to estimate $A_1$.
Fix $x\in M$ and an orthonormal basis $\zeta_1,\dots,\zeta_m\in E_x$,
where $m=\dim E_x$.
Then $\phi(v)=\sum_j \big(f_j(v) + i \,g_j(v)\big) \zeta_j$ 
for some functions $f_j,g_j\co \mathcal{B}_\rho(0)\to\R$,
$j=1,\dots,m$.
Then the formula \eqref{e:A1a} takes the form
$$
 A_1(x) \cdot w = -\theta(x)^{-1} \frac{n+2}{\nu_n\rho^{n+2}} 
 \sum_{j=1}^m 
 \bigl( \left\langle a_j , w \right\rangle
 + i \left\langle b_j , w \right\rangle
 \bigr) \zeta_j,
$$
for any $w\in T_xM$,
where $a_j,b_j\in T_xM$ are given by
$$
 a_j = \int_{\mathcal{B}_\rho(0)} f_j(v) v J_x(v)\, dv,\qquad
 b_j = \int_{\mathcal{B}_\rho(0)} g_j(v) v J_x(v)\, dv .
$$
Hence by using \eqref{e:norm-nabla},
$$
  \|A_1(x)\|^2 = \left(\theta(x)^{-1} \frac{n+2}{\nu_n\rho^{n+2}}\right)^2 \ 
  \sum_{j=1}^m (a_j^2+b_j^2).
$$
Applying the Sublemma to $f_j J_x$ and $g_j J_x$ in place of $f$, we obtain
$$
 a_j^2 + b_j^2 
 \le \frac{\nu_n\rho^{n+2}}{n+2} 
 \int_{\mathcal{B}_\rho(0)} \big(f_j(v)^2+g_j(v)^2\big) J_x(v)^2 \,dv .
$$
Thus,
\begin{eqnarray*}
 \|A_1(x)\|^2 &\le& \theta(x)^{-2}
 \frac{n+2}{\nu_n\rho^{n+2}} 
 \int_{\mathcal{B}_\rho(0)} |\phi(v)|^2 J_x(v)^2 \,dv \\
 &\le& \theta(x)^{-2} \frac{n+2}{\nu_n\rho^{n+2}} J_{\max}
 \int_{B_\rho(x)} |\G_{xy}(u)|^2 \,dy ,
\end{eqnarray*}
where we used
again the substitution $y=\exp_x(v)$ and the definition of $\phi$ in the last inequality.

By \eqref{e:theta-bound} and \eqref{e:rauch}, 
we have $\theta(x)^{-2}J_{\max}\le 1+C_n\KM\rho^2$.
Hence,
$$
 \|A_1(x)\|^2 \le (1+C_n\KM\rho^2) \frac{n+2}{\nu_n\rho^{n+2}} 
 \int_{B_\rho(x)} |\G_{xy}(u)|^2 \,dy .
$$
Integrating over $M$ yields that
\be\label{e:A1bound}
 \|A_1\|_{L^2} \le (1+C_n\KM\rho^2) \sqrt{\frac{n+2}{\nu_n\rho^{n+2}} D^{\rho}(u)}\, .
\ee
From $\nabla Iu=A_1+A_2+A_3$, the lemma
follows from \eqref{e:A1bound}, \eqref{e:A2bound}, and \eqref{e:A3bound}.
\end{proof}

The lower bound for $\til\la_k$ is a consequence of Lemma \ref{l:Iu-lower-bound} and Lemma \ref{l:nablaIu-upper-bound}.

\begin{proposition}\label{p:wtla-lower-bound}
For every $k\in\N_+$ satisfying $\wt\la_k\le  \sigma_n \rho^{-2}$, we have
$$
 \la_k^{1/2} \le \left(1+C_n\KM\rho^2+\wt\la_k \rho^2 \right)\wt\la_k^{1/2} + C_n\KE\rho \, .
$$
\end{proposition}

\begin{proof}
Let $u_1,\dots, u_k\in L^2(M,E)$ be the first $k$ eigen-sections of $\Delta^\rho$
corresponding to $\wt\la_1,\dots,\wt\la_k$.
Let $\wt L$ be the linear span of $u_1,\dots,u_k$.
Then $\wt L$ realizes the infimum in the min-max formula \eqref{e:minmax-tilde}.
Hence for all $u\in\wt L$,
\be\label{e:Dlewtla}
 D^{\rho}(u) \le \frac{\nu_n\rho^{n+2}}{n+2} \wt\la_k \|u\|^2_{L^2}.
\ee
This and Lemma \ref{l:Iu-lower-bound} imply that
\be\label{e:Iu2}
  \|Iu\|^2_{L^2} \ge \frac{J_{\min}}{J_{\max}}\, \|u\|^2_{L^2} - \frac{n+2}{2\nu_n\rho^n} D^{\rho}(u) 
  \ge \left( \frac{J_{\min}}{J_{\max}} - \frac12\rho^2 \wt\la_k \right) \|u\|^2_{L^2} .
\ee
For $\rho^2 \wt\la_k\le 1$ and $\frac{J_{\min}}{J_{\max}}>\frac12$,
the left-hand side of the inequality above is positive for all $u\in\wt L\setminus\{0\}$.
In particular, $I$ is injective on $\wt L$. 

Define $L=I(\wt L)\subset C^{0,1}(M,E)$. Since $I|_{\wt L}$ is injective,
we have $\dim L=\dim\wt L=k$.
Now the min-max formula \eqref{e:minmax-smooth}, Lemma \ref{l:nablaIu-upper-bound} and \eqref{e:Iu2} imply that
\begin{eqnarray*}
\la_k^{1/2} &\le& \sup_{ u \in L\setminus\{0\}} \frac{\|\nabla u\|_{L^2}}{\|u\|_{L^2}}
= \sup_{ u \in \wt L\setminus\{0\}} \frac{\|\nabla Iu\|_{L^2}}{\|Iu\|_{L^2}} \\
&\le& \frac
{
 (1+C_n\KM\rho^2) \sqrt{\frac{n+2}{\nu_n\rho^{n+2}} D^{\rho}(u)}
  + C_nK_E\,\rho \,\|u\|_{L^2}
} 
{
\left(\frac{J_{\min}}{J_{\max}} - \frac12\rho^2 \wt\la_k \right)^{\frac12} \|u\|_{L^2}
}.
\end{eqnarray*}
This and \eqref{e:Dlewtla} imply that
\be\label{e:lak12}
 \la_k^{1/2}
 \le \frac
  { (1+C_n\KM\rho^2) \wt\la_k^{1/2}   + C_nK_E\,\rho } 
  { \left(\frac{J_{\min}}{J_{\max}} - \frac12\rho^2 \wt\la_k \right)^{\frac12} } .
\ee
Using the Jacobian estimate \eqref{e:rauch}, one sees that
$\frac{J_{\min}}{J_{\max}} \ge 1-C_n\KM\rho^2$,
which implies that
$$
 \left(\frac{J_{\min}}{J_{\max}} - \frac12\rho^2 \wt\la_k \right)^{-\frac12}
 \le 1 + C_n\KM\rho^2 + \wt\la_k \rho^2 .
$$
Then the proposition follows.
\end{proof}

\begin{proof}[Proof of Theorem \ref{t:smoothcase}]
The estimate directly follows from Proposition \ref{p:wtla-upper-bound} and Proposition \ref{p:wtla-lower-bound}, after converting all error terms involving $\til\lambda_k$ to $\lambda_k$ by using Proposition \ref{p:wtla-upper-bound}.
\end{proof}

\section{Discretization of the connection Laplacian}
\label{sec:graph}

In this section we prove Theorem \ref{t:graphcase}. Let $M^n$ be a compact, connected Riemannian manifold of dimension $n$ without boundary, and let $E$ be a smooth Euclidean (or Hermitian) vector bundle over $M$ equipped with a smooth Euclidean (or Hermitian) connection $\nabla$. Suppose $P=\{P_{xy}\}$ is the $\rho$-connection given by the parallel transport canonically associated with the connection $\nabla$. Recall that $P$ is unitary and symmetric.
Let $\Gamma=(X_{\varepsilon},d|_{X_{\varepsilon}},\mu)$ (short for $\Gamma_{\ep}$) be the discrete metric-measure space defined in Definition \ref{def-graph}, where $X_{\varepsilon}=\{x_i\}_{i=1}^N$ is a finite $\varepsilon$-net in $M$ for $\varepsilon\ll\rho$. 
We consider the $\rho$-connection Laplacian \eqref{e:mm-laplace} on this discrete metric-measure space $\Gamma$, acting on the restriction of the vector bundle $E$ onto $X_{\varepsilon}$. 

The vector bundle $E$ restricted onto $X_{\varepsilon}$ is equipped with the norm
\begin{equation}\label{discrete-norm}
\|\bar{u}\|_{\Gamma}^2=\sum_{i=1}^N \mu_i |\bar{u}(x_i)|^2,
\end{equation}
for $\bar{u}\in L^2(X_{\varepsilon},E|_{X_{\varepsilon}})$. Choosing the weights $\alpha(x)=1$ and
$\beta(x,y)=\frac{2(n+2)}{\nu_n\rho^{n+2}}$ in \eqref{e:mm-laplace} the same as in the case of smooth connections, the graph connection Laplacian $\Delta^{\rho}_{\Gamma}$ is given by
\begin{equation}\label{graph-laplacian}
\Delta^{\rho}_{\Gamma} \bar{u}(x_i)=\frac{2(n+2)}{\nu_n \rho^{n+2}} \sum_{d(x_i,x_j)<\rho} \mu_j \big(\bar{u}(x_i)-P_{x_i x_j}\bar{u}(x_j)\big),
\end{equation}
and its energy \eqref{def-energy} is given by
\begin{equation}\label{def-discrete-energy}
\|\delta \bar{u}\|^2:=\frac{n+2}{\nu_n \rho^{n+2}}\sum_i\sum_{j: d(x_i,x_j)<\rho} \mu_i \mu_j \big|\bar{u}(x_i)-P_{x_i x_j}\bar{u}(x_j) \big|^2.
\end{equation}
Denote by $\til\lambda_k(\Gamma)$ the $k$-th eigenvalue of $\Delta_{\Gamma}^{\rho}$. Our goal is to prove that $\til\lambda_k(\Gamma)$ approximates the eigenvalue $\lambda_k$ of the connection Laplacian $\Delta$ for every $k$, as $\rho+\frac{\varepsilon}{\rho}\to 0$. In the light of what we have already done in Section \ref{sec:smooth}, we only need to obtain a few more estimates.

\medskip
For the upper bound for $\til\lambda_k(\Gamma)$, we follow Section $4$ in \cite{BIK1} and define the following discretization operator $Q: L^2(M,E) \to L^2(X_{\varepsilon},E|_{X_{\varepsilon}})$ by
\begin{equation}\label{def-Q}
Qu(x_i)=\frac{1}{\mu_i}\int_{V_i} P_{x_i y}u(y)\,dy.
\end{equation}
Define an extension operator $Q^{\ast}: L^2(X_{\varepsilon},E|_{X_{\varepsilon}}) \to L^2(M,E)$ by
\begin{equation}\label{def-Qstar}
Q^{\ast}\bar{u}(y)=\sum_{i=1}^{N} P_{y x_i}\bar{u}(x_i) 1_{V_i}(y),
\end{equation}
where $1_{V_i}$ denotes the characteristic function of the set $V_i$. 
Note that $Q\circ Q^{\ast}=Id_{L^2(X_{\varepsilon},E|_{X_{\varepsilon}})}$.
The energy of $Qu$ for $u\in L^2(M,E)$ is given by
\begin{equation}\label{discrete-energy-Qu}
\|\delta(Qu)\|^2=\frac{n+2}{\nu_n \rho^{n+2}}\sum_i\sum_{j: d(x_i,x_j)<\rho} \mu_i \mu_j \big|Qu(x_i)-P_{x_i x_j}Qu(x_j) \big|^2.
\end{equation}

To control the upper bound for $\til\lambda_k(\Gamma)$, we need to estimate $\|Qu\|_{\Gamma}$ and $\|\delta(Qu)\|$. We start with the following lemma as an application of Lemma \ref{l:parallel}.

\begin{lemma}\label{changingpath}
Let $\rho<r_{inj}(M)/2$, $\varepsilon<\rho/4$, and $x_i,x_j,y,z\in M$ be given satisfying $d(x_i,x_j)<\rho,\,d(y,x_i)<\ep,\, d(z,x_j)<\ep$. Then for any $v\in E_z$, we have
$$\big|P_{y x_i} P_{x_i z}v-P_{yz}v \big|\leq K_E (\rho+2\ep) \ep |v|,$$
and
$$\big| P_{y x_i}P_{x_i x_j}P_{x_j z}v -P_{yz}v \big|\leq 2 K_E (\rho+2\ep)\varepsilon |v|.$$
\end{lemma}
\begin{proof}
Let $V_{z,v}(x)=P_{xz}(v)$, and let $\gamma_{y,x_i}: [0,d(y,x_i)] \to M$ be the unique minimizing geodesic from $y$ to $x_i$ with arclength parametrization. By definition,
$$\nabla_s V_{z,v}\big(\gamma_{y,x_i}(s)\big)=\lim_{s'\to 0}\frac{P_{s,s+s'}V_{z,v}\big(\gamma_{y,x_i}(s+s')\big)-V_{z,v}\big(\gamma_{y,x_i}(s)\big)}{s'},$$
where $P_{s,s+s'}$ denotes the parallel transport from $\gamma_{y,x_i}(s+s')$ to $\gamma_{y,x_i}(s)$ along the geodesic $\gamma_{y,x_i}$. Apply $P_{0,s}$ to both sides:
$$P_{0,s}\Big(\nabla_s V_{z,v}\big(\gamma_{y,x_i}(s)\big)\Big)=\lim_{s'\to 0}\frac{P_{0,s+s'}V_{z,v}\big(\gamma_{y,x_i}(s+s')\big)-P_{0,s} V_{z,v}\big(\gamma_{y,x_i}(s)\big)}{s'}.$$
Observe that $P_{0,\cdot} V_{z,v}\big(\gamma_{y,x_i}(\cdot)\big)$ is a curve in $E_y$. Since $P$ is unitary, the formula above shows that the tangent vectors of this $E_y$-curve have lengths bounded by $K_E (\rho+2\ep)|v|$ due to \eqref{error}. Thus,
$$\big|P_{y x_i}V_{z,v}(x_i)-V_{z,v}(y) \big|=\Big|\int_0^{d(x_i,y)} P_{0,s}\Big(\nabla_s V_{z,v}\big(\gamma_{y,x_i}(s)\big)\Big) \,ds\, \Big| \leq K_E (\rho+2\ep)\ep |v|.$$
Then the first conclusion directly follows from the definition $V_{z,v}(x)=P_{xz}(v)$.

\smallskip
The second conclusion can be derived using the first conclusion. Namely, we apply the first conclusion with $x_j,z,y$ in place of $y,x_i,z$, and $P_{yz}v$ in place of $v$:
$$\big| P_{x_jz} P_{zy} (P_{yz}v)-P_{x_j y}(P_{yz}v) \big| \leq K_E (\rho+2\varepsilon)\varepsilon |P_{yz}v|=  K_E (\rho+2\varepsilon)\varepsilon |v|.$$
Since $P$ is symmetric and unitary, the inequality above is equivalent to
\begin{equation} \label{eq-loop1}
\big| P_{y x_j}P_{x_jz} v-P_{yz}v \big| \leq   K_E (\rho+2\varepsilon)\varepsilon |v|.
\end{equation}
Applying the first conclusion with $y,x_i,x_j$ in place of $y,x_i,z$, and $P_{x_j z}v$ in place of $v$ gives
\begin{equation}\label{eq-loop2}
\big| P_{y x_i}P_{x_i x_j} (P_{x_j z}v)-P_{y x_j} (P_{x_j z} v) \big| \leq   K_E (\rho+2\varepsilon)\varepsilon |P_{x_j z}v|=K_E (\rho+2\varepsilon)\varepsilon |v|.
\end{equation}
Thus the second conclusion follows from \eqref{eq-loop1}, \eqref{eq-loop2} and the triangle inequality.
\end{proof}

The following two lemmas enable us to obtain an upper bound for $\til\lambda_k(\Gamma)$.

\begin{lemma}\label{discrete-L2}
For any $u\in L^2(M,E)$, we have 
$$\big| \|u\|_{L^2}-\|Qu\|_{\Gamma} \big|^2 \leq\frac{C}{\nu_n(\rho-\varepsilon)^n}D^{\rho}(u)+C K_E^2\rho^4\|u\|_{L^2}^2.$$
\end{lemma}
\begin{proof}
Observe that $Q^{\ast}$ preserves the norms. Hence,
$$\big| \|u\|_{L^2}-\|Qu\|_{\Gamma} \big|= \big| \|u\|_{L^2}-\|Q^{\ast}Qu\|_{L^2} \big| \leq \|u-Q^{\ast}Qu\|_{L^2}.$$
By the definitions of $Q$ and $Q^{\ast}$,
\begin{eqnarray*}
\|u-Q^{\ast}Qu\|_{L^2}^2 &=& \sum_{i=1}^N \int_{V_i} \big|u(x)-P_{x x_i}Qu(x_i) \big|^2 dx \\
&=& \sum_{i=1}^N \int_{V_i} \Big|u(x)-\frac{1}{\mu_i}\int_{V_i} P_{x x_i}P_{x_i y}u(y)\, dy \Big|^2 dx.
\end{eqnarray*}
By the Cauchy-Schwarz inequality, we have
$$\|u-Q^{\ast}Qu\|_{L^2}^2 \leq \sum_{i=1}^N \frac{1}{\mu_i} \int_{V_i}\int_{V_i} \big|u(x)-P_{x x_i}P_{x_i y}u(y) \big|^2 dx dy. $$
Since $V_i\subset B_{\varepsilon}(x_i)$, then Lemma \ref{changingpath} yields that
\begin{eqnarray*}
\|u-Q^{\ast}Qu\|_{L^2}^2 \leq \sum_{i=1}^N \frac{1}{\mu_i} \int_{V_i}\int_{V_i} \big(|u(x)-P_{xy}u(y)|+2K_E\varepsilon^2 |u(y)| \big)^2 dxdy .
\end{eqnarray*}

To deal with the first term, we follow the proof of Lemma $3.4$ in \cite{BIK1}. We fix $x,y\in V_i$ and consider the set $U=B_{\rho}(x)\cap B_{\rho}(y)$. Observe that $U$ contains the ball of radius $\rho-|xy|/2\geq\rho-\varepsilon$ centered at the midpoint between $x$ and $y$. 
Hence we have $\textrm{vol}(U)\geq C \nu_n (\rho-\varepsilon)^n$ by \eqref{e:rauch}.
Recall that $P$ is unitary and symmetric. Then for every $z\in U$, we have 
\begin{eqnarray*}
|u(x)-P_{xy}u(y)| &\leq& |u(x)-P_{xz}u(z)|+|P_{xz}u(z)-P_{xy}u(y)| \\
&=& |u(x)-P_{xz}u(z)|+|u(z)-P_{zx}P_{xy}u(y)| \\
&\leq& |u(x)-P_{xz}u(z)|+|u(z)-P_{zy}u(y)|+K_E\rho^2 |u(y)|,
\end{eqnarray*}
where we applied Lemma \ref{changingpath} in the last inequality. 
Then
\begin{eqnarray*}
|u(x)-P_{xy}u(y)|^2 &\leq& \frac{C}{\vol(U)} \int_{U} \Big(|u(x)-P_{xz}u(z)|^2+|u(y)-P_{yz}u(z)|^2+K_E^2 \rho^4 |u(y)|^2 \Big) dz \\
&\leq& \frac{C}{\vol(U)}\big(F(x)+F(y)\big) +C K_E^2 \rho^4 |u(y)|^2,
\end{eqnarray*}
where $F(x)=\int_{B_{\rho}(x)}|u(x)-P_{xz}u(z)|^2 dz$. Hence by definition \eqref{e:defD}, we obtain
\begin{eqnarray*}
\|u-Q^{\ast}Qu\|_{L^2}^2 &\leq& \sum_{i=1}^N \frac{C}{\mu_i} \int_{V_i}\int_{V_i} \frac{1}{\vol(U)}\big(F(x)+F(y)\big)dxdy + CK_E^2 \rho^4 \|u\|_{L^2}^2 \\
&=& \frac{C}{\vol(U)} D^{\rho}(u) +CK_E^2 \rho^4 \|u\|_{L^2}^2 \leq \frac{C}{\nu_n(\rho-\ep)^n} D^{\rho}(u) +CK_E^2 \rho^4 \|u\|_{L^2}^2.
\end{eqnarray*}
\end{proof}

\begin{lemma}\label{discrete-energy}
For any $u\in L^2(M,E)$, we have 
$$\|\delta(Qu)\|^2\leq \frac{n+2}{\nu_n \rho^{n+2}}(1+4\rho^2) D^{\rho+2\varepsilon}(u)+ C_n K_E^2 (\frac{\varepsilon}{\rho})^2 \|u\|_{L^2}^2.$$
\end{lemma}

\begin{proof}
The definition of $Q$ yields that
$$Qu(x_i)-P_{x_i x_j}Qu(x_j)=\frac{1}{\mu_i \mu_j}\int_{V_i}\int_{V_j} \big( P_{x_i y}u(y)-P_{x_i x_j}P_{x_j z}u(z)\big) dy dz.$$
Then by the Cauchy-Schwarz inequality and the fact that $P$ is unitary and symmetric,
\begin{eqnarray*}
\big|Qu(x_i)-P_{x_i x_j}Qu(x_j)\big|^2&\leq&\frac{1}{\mu_i \mu_j}\int_{V_i}\int_{V_j} \big| P_{x_i y}u(y)-P_{x_i x_j}P_{x_j z}u(z) \big|^2 dy dz\\
&=& \frac{1}{\mu_i \mu_j}\int_{V_i}\int_{V_j} \big| u(y)-P_{y x_i}P_{x_i x_j}P_{x_j z}u(z) \big|^2 dy dz.
\end{eqnarray*}
The parallel transport appeared in the quantity above goes through the path $[z x_j x_i y]$, while what we need is to go through the minimizing geodesic $[zy]$. Thus \eqref{discrete-energy-Qu} and Lemma \ref{changingpath} imply that
\begin{eqnarray*}
\|\delta(Qu)\|^2 &\leq& \frac{n+2}{\nu_n \rho^{n+2}} \sum_i\sum_{j: d(x_i,x_j)<\rho} \int_{V_i}\int_{V_j} \big| u(y)-P_{y x_i}P_{x_i x_j}P_{x_j z}u(z) \big|^2 dy dz \\
&\leq& \frac{n+2}{\nu_n \rho^{n+2}} \sum_i\sum_{j: d(x_i,x_j)<\rho} \int_{V_i}\int_{V_j} \big(| u(y)-P_{yz}u(z) |+ 4 K_E\rho\varepsilon |u(z)| \big)^2 dy dz \\
&\leq& \frac{n+2}{\nu_n \rho^{n+2}} \sum_i\sum_{j: d(x_i,x_j)<\rho} \int_{V_i}\int_{V_j} \Big((1+4\rho^2)| u(y)-P_{yz}u(z) |^2+CK_E^2\varepsilon^2|u(z)|^2 \Big) dy dz.
\end{eqnarray*}
Here the last inequality above used the inequality that 
$$2K_E\rho\varepsilon |u(z)| \cdot | u(y)-P_{yz}u(z) | \leq \rho^2| u(y)-P_{yz}u(z) |^2+K_E^2\varepsilon^2|u(z)|^2.$$ 
Since $\bigcup_{j: d(x_i, x_j)<\rho}V_j \subset B_{\rho+2\varepsilon}(y)$ for $y\in V_i$, we have
\begin{eqnarray*}
\|\delta(Qu)\|^2 &\leq& \frac{n+2}{\nu_n \rho^{n+2}}(1+4\rho^2) \int_{M}\int_{B_{\rho+2\varepsilon}(y)} | u(y)-P_{yz}u(z) |^2 dydz +C_n K_E^2 \frac{\varepsilon^2 (\rho+2\varepsilon)^n}{\rho^{n+2}}\|u\|_{L^2}^2 \\
&\leq& \frac{n+2}{\nu_n \rho^{n+2}}(1+4\rho^2) D^{\rho+2\varepsilon}(u)+ C_n K_E^2(\frac{\varepsilon}{\rho})^2\|u\|_{L^2}^2.
\end{eqnarray*}
\end{proof}

The lower bound for $\til\lambda_k(\Gamma)$ almost immediately follows from Lemma \ref{l:Iu-lower-bound} and Lemma \ref{l:nablaIu-upper-bound}, since these two lemmas hold for any $L^2$ section. For any $\bar{u}\in L^2(X_{\varepsilon},E|_{X_{\varepsilon}})$, we consider $Q^{\ast}\bar{u} \in L^2(M,E)$ and apply those two lemmas to $Q^{\ast}\bar{u}$. Recall that $\|Q^{\ast}\bar{u}\|_{L^2}^2=\|\bar{u}\|^2_{\Gamma}$. The only part left is to estimate $D^{\rho}(Q^{\ast}\bar{u})$ in terms of $\|\delta \bar{u}\|^2$.

\begin{lemma}\label{discrete-lower}
For any $\bar{u}\in L^2(X_{\varepsilon},E|_{X_{\varepsilon}})$, we have
$$D^{\rho-2\varepsilon}(Q^{\ast}\bar{u}) \leq \frac{\nu_n \rho^{n+2}}{n+2}(1+4\rho^2) \|\delta\bar{u}\|^2 + C_n K_E^2\varepsilon^2(\rho+\varepsilon)^n \|\bar{u}\|_{\Gamma}^2.$$
\end{lemma}

\begin{proof}
Since $B_{\rho-2\varepsilon}(y) \subset \bigcup_{j: d(x_i, x_j)<\rho}V_j $ for $y\in V_i$, we have
\begin{eqnarray*}
D^{\rho-2\varepsilon}(Q^{\ast}\bar{u}) &=& \int_M \int_{B_{\rho-2\varepsilon}(y)} |Q^{\ast}\bar{u}(y)-P_{yz}Q^{\ast}\bar{u}(z)|^2 dz dy \\
&\leq& \sum_i \sum_{j: d(x_i,x_j)<\rho} \int_{V_i} \int_{V_j} |Q^{\ast}\bar{u}(y)-P_{yz}Q^{\ast}\bar{u}(z)|^2 dz dy.
\end{eqnarray*}
By the definition of $Q^{\ast}$, for any $y\in V_i,\,z\in V_j$,
$$Q^{\ast}\bar{u}(y)-P_{yz}Q^{\ast}\bar{u}(z) =P_{y x_i}\bar{u}(x_i)-P_{yz}P_{z x_j}\bar{u}(x_j) .$$
Since $P$ is unitary and symmetric, Lemma \ref{changingpath} implies that
\begin{eqnarray*}
\big|Q^{\ast}\bar{u}(y)-P_{yz}Q^{\ast}\bar{u}(z)\big|^2 &=& \big|\bar{u}(x_i)-P_{x_i y}P_{yz}P_{z x_j}\bar{u}(x_j) \big|^2 \\
&\leq& \big(|\bar{u}(x_i)-P_{x_i x_j}\bar{u}(x_j)|+4K_E \rho\varepsilon |\bar{u}(x_j)| \big)^2 \\
&\leq& (1+4\rho^2) \big|\bar{u}(x_i)-P_{x_i x_j}\bar{u}(x_j) \big|^2+ CK_E^2\varepsilon^2 |\bar{u}(x_j)|^2.
\end{eqnarray*}
Integrating the last inequality over $V_i,V_j$, by definition \eqref{def-discrete-energy}, we obtain
\begin{eqnarray*}
D^{\rho-2\varepsilon}(Q^{\ast}\bar{u}) &\leq& (1+4\rho^2) \sum_{i,j} \mu_i \mu_j \big|\bar{u}(x_i)-P_{x_i x_j}\bar{u}(x_j) \big|^2 +  CK_E^2 \varepsilon^2 \sum_{i,j}\mu_i \mu_j  |\bar{u}(x_j)|^2  \\
&\leq& \frac{\nu_n \rho^{n+2}}{n+2}(1+4\rho^2) \|\delta\bar{u}\|^2 + C_n K_E^2\varepsilon^2(\rho+\varepsilon)^n \|\bar{u}\|_{\Gamma}^2,
\end{eqnarray*}
where the last inequality used the fact that $\sum_{i: d(x_i,x_j)<\rho} \mu_i \leq \vol(B_{\rho+\ep}(x_j))$.
\end{proof}

\begin{proof}[Proof of Theorem \ref{t:graphcase}]
The upper bound for $\til\lambda_k(\Gamma)$ follows from Lemma \ref{discrete-L2}, Lemma \ref{discrete-energy} and Lemma \ref{l:Du-upper}. The lower bound follows from Lemma \ref{discrete-lower}, Lemma \ref{l:Iu-lower-bound} and Lemma \ref{l:nablaIu-upper-bound}. The calculations are straightforward, similar to Proposition \ref{p:wtla-lower-bound}.
\end{proof}

\end{document}